\documentclass[11pt,reqno]{amsart}
\usepackage{amssymb}
\usepackage{graphicx}

{\theoremstyle{plain}
 \newtheorem{theorem}{Theorem}
 \newtheorem{corollary}{Corollary}
  \newtheorem{remark}{Remark}

\newtheorem{lemma}{Lemma}
}

\DeclareMathOperator{\res}{Res}
\DeclareMathOperator{\re}{Re}

\begin{document}

\title[Mean number of divisors for rough, dense, practical numbers]
{The mean number of divisors for rough, dense and practical numbers}

\begin{abstract}
We give asymptotic estimates for the mean number of divisors of integers without small prime factors, integers
with bounded ratios of consecutive divisors, and for practical numbers. In the last case, this confirms a conjecture of 
Margenstern.
\end{abstract}

\author{Andreas Weingartner}
\address{ 
Department of Mathematics,
351 West University Boulevard,
 Southern Utah University,
Cedar City, Utah 84720, USA}
\email{weingartner@suu.edu}
\date{May 10, 2021}
\dedicatory{Dedicated to Maurice Margenstern for his insightful conjectures on practical numbers}
\maketitle

\section{Introduction}

Let $\tau(n)$ be the number of positive divisors of $n$. 
We explore the average size of $\tau(n)$ as $n$ runs through one of several integer sequences.
First, we look at the $y$-rough numbers, i.e. integers with no prime factors $\le y$.
Let $P^-(n)$ denote the smallest prime factor of $n\ge 2$ and put $P^-(1)=\infty$. 
Define the function $\xi(u)$ by $\xi(u)=0$ for $u<1$ and 
\begin{equation}\label{xidef}
u \xi(u) = 2+2 \int_{1}^{u-1} \xi(t)\,  dt \qquad (u\ge 1).
\end{equation}

\begin{theorem}\label{thm1}
Uniformly for $x\ge 1$, $y\ge 2$, $u=\log x /\log y$,
\begin{multline*}
S(x,y):=\sum_{n \le x \atop P^-(n)>y} \tau(n) =1+ x (\log x) \prod_{p\le y}\left(1-\frac{1}{p}\right)^2 \\
+ \frac{x}{\log y}\left\{\xi(u) - ue^{-2\gamma}-\left.\frac{2y}{x}\right|_{x\ge y}
+O\left(\frac{1}{\log y}\right)\right\}
\end{multline*}
and
$$
\xi(u)= (u+ 2)e^{-2\gamma} + O\left(u^{-u}\right) \qquad (u\ge 1).
$$
\end{theorem}

With its three main terms of magnitude $x u /\log y$, one of which is negative, Theorem \ref{thm1} is somewhat difficult to 
make sense of. However, this form is preferable when  we apply it to prove Theorems \ref{thm2} and \ref{thmgen}. 
The estimate for $\xi(u)$ and Mertens' formula allow us to simplify Theorem \ref{thm1} as follows.

\begin{corollary}\label{cor0}
We have
$$
S(x,y)=  x \log (x y^2)\prod_{p\le y}\left(1-\frac{1}{p}\right)^2
\left\{1+O\left(\frac{1}{\log x} +u^{-u}\right) \right\} \quad (x\ge y\ge 2), 
$$
$$
S(x,y)=   \frac{x\xi(u)-2y}{\log y}\left\{1+O\left(\frac{1}{\log x}+e^{-\sqrt{\log y}}\right) \right\} 
  \quad (x\ge 2y\ge 4).
$$
\end{corollary}

Dividing these estimates by that of  Lemma \ref{PhiLemma} for $\Phi(x,y)$, the number of $y$-rough integers $n\le x$,
we obtain the following results for the mean value of $\tau(n)$.
Throughout, $\omega(u)$ denotes Buchstab's function.
\begin{corollary}\label{cor1}
We have
$$
\frac{S(x,y)}{\Phi(x,y)} = \log (x y^2)\prod_{p\le y}\left(1-\frac{1}{p}\right)\left\{1+O\left(\frac{1}{\log x} +u^{-u}\right) \right\} \quad (x\ge y\ge 2), 
$$
\begin{equation*}\label{cor1main}
\frac{S(x,y)}{\Phi(x,y)} =   \frac{\xi(u)}{\omega(u)}\left\{1+O\left(\frac{1}{\log x}+e^{-\sqrt{\log y}}\right) \right\} 
  \quad (x\ge 2y\ge 4).
\end{equation*}
 \end{corollary}
 
 Corollary \ref{cor1} shows that, for fixed $u$, the average of $\tau(n)$ over $x^{1/u}$-rough integers $n\le x$
 is asymptotic to $\xi(u)/\omega(u)$, as $x\to \infty$. 
 The graphs of $\xi(u)$ and $\xi(u)/\omega(u)$ are shown in Figure \ref{figure1}.
 
 \begin{figure}[h]
\begin{center}
\includegraphics[height=6cm,width=10cm]{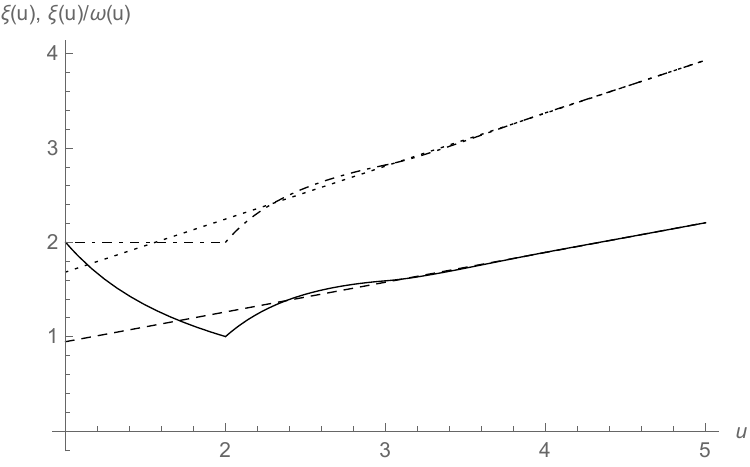}
\caption{A graph of the function $\xi(u)$ (solid) and its asymptote  $(u+2)e^{-2\gamma}$ (dashed),
and the ratio $\xi(u)/\omega(u)$ (dot-dashed) and its asymptote $(u+2)e^{-\gamma}$ (dotted).}
\label{figure1}
\end{center}
\end{figure}

We will see in Section \ref{SecProofThm1} that Theorem \ref{thm1} follows without much difficulty from 
Lemma \ref{PhiLemma}.
Our main motivation for including Theorem \ref{thm1} is its role in the proofs of the following results. 

Next, we consider integers with bounded ratios of consecutive divisors. 
We say that $n$ is $t$-dense (or $t$-densely divisible) if the ratios of consecutive divisors of $n$ do not exceed $t$. 
Let $\mathcal{D}(x,t)$ denote the set of $t$-dense integers $n\le x$ and write $D(x,t)= |\mathcal{D}(x,t)|$.
Let $\lambda(v)$ be given by $\lambda(v)=0$ for $v<0$ and 
\begin{equation}\label{lambdadef}
\lambda(v)=v-\int_0^{(v-1)/2} \lambda(u)\, \xi\left(\frac{v-u}{u+1}\right)
\frac{du}{u+1} \quad (v\ge 0).
\end{equation}

\begin{theorem}\label{thm2}
Uniformly for $x\ge 1$, $t\ge 2$, $v=\log x /\log t$,
$$
T(x,t):=\sum_{n \in \mathcal{D}(x,t) } \tau(n) = x \alpha_t (\log t) \lambda(v)  + O(x),
$$
where $0<\alpha_0 \le \alpha_t = 1+O(1/\log t)$ and 
\begin{equation}\label{lamapp}
\lambda(v) = \lambda_0 (v+1)^\delta + \lambda_1 (v+1)^{-1} + O( (v+1)^{-1.962}),
\end{equation}
$$
\delta=0.7136125..., \quad \lambda_0=1.118192..., \quad \lambda_1
=\frac{2}{3e^{-2\gamma}-2}
=-1.897014...
$$
The constants $\delta$ and $\lambda_0$ are defined in Lemma \ref{LambdaLem}.
\end{theorem}

 \begin{figure}[h]
\begin{center}
\includegraphics[height=6cm,width=10cm]{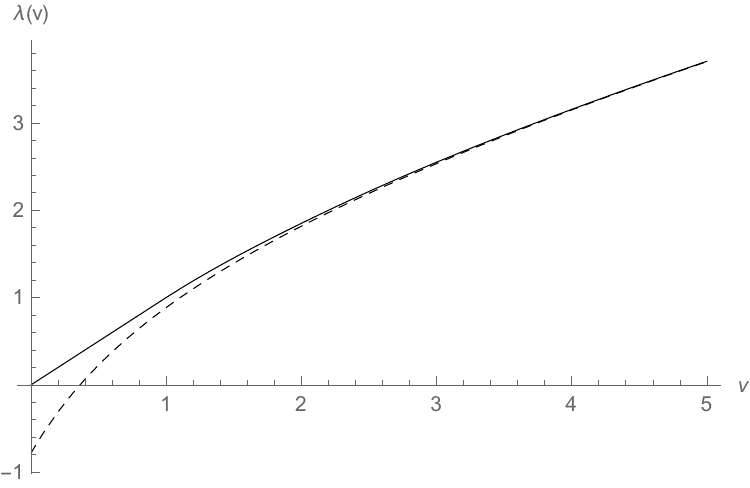}
\caption{The function $\lambda(v)$ (solid) and its approximation $ \lambda_0 (v+1)^\delta + \lambda_1 (v+1)^{-1}$ (dashed).}
\label{figure2}
\end{center}
\end{figure}

If one is only concerned with the order of magnitude of $T(x,t)$, Theorem \ref{thm2} and the estimate $D(x,t)\asymp x/v$ 
of Saias \cite[Thm. 1]{Saias1} imply the following simple formulas.
\begin{corollary}\label{ordmag}
Uniformly for $x\ge t\ge 2$, $v=\log x /\log t$,
$$
T(x,t) \asymp x v^\delta \log t, \qquad
\frac{T(x,t)}{D(x,t)}\asymp v^\delta \log x.
$$
\end{corollary}
Corollary \ref{ordmag} shows how the average size of $\tau(n)$ varies from $\asymp (\log x)^{1+\delta}=(\log x)^{1.713...}$
in the case of $2$-dense integers, to $\asymp \log x $ when $t=x$ and $\mathcal{D}(x,t)$ contains all natural numbers up to $x$.
As in the case of natural numbers, the typical size of $\tau(n)$ is much smaller than the average size. 
Theorem 1 of \cite{OMB} shows that almost all $n \in \mathcal{D}(x,t)$ satisfy
$$
\tau(n)^{1+o(1)}= v^{ C\log 2} (\log t)^{\log 2}, \quad C=(1-e^{-\gamma})^{-1}=2.280...,
$$ 
which in the case of $2$-dense integers is $\asymp (\log x)^{C\log 2} = (\log x)^{1.580...}$. 

For more precise information, we combine the estimates for $T(x,t)$, $\lambda(v)$ and $\alpha_t$ in Theorem \ref{thm2}, to obtain
the following three asymptotic estimates, which reveal the behavior of $T(x,y)$ in each of the three cases:
(i) fixed $t$, (ii) $t\to \infty$,
$v \to \infty$ and (iii) fixed $v$.  
\begin{corollary}\label{cor2}
Uniformly for $x\ge t\ge 2$, $v=\log x /\log t$,
$$
T(x,t)=a_t x (\log xt)^\delta \left\{ 1+ O\left(\frac{1}{v^\delta \log t}+\frac{1}{v^{1+\delta}} \right) \right\},
$$
where $a_t=\lambda_0 \alpha_t (\log t)^{1-\delta}$, 
$$
T(x,t)=\lambda_0 x (v+1)^\delta (\log t)
 \left\{ 1+ O\left(\frac{1}{\log t}+\frac{1}{v^{1+\delta}} \right) \right\},
$$
$$
T(x,t)= x (\log t) \lambda(v) \left\{ 1+ O\left(\frac{1}{\log t}\right) \right\}.
$$
\end{corollary}

With the corresponding three estimates for $D(x,t)$ in \cite[Corollaries 1.1-1.3]{PDD}, we obtain
the following formulas for the mean value of $\tau(n)$ over $t$-dense integers $n\le x$. 
\begin{corollary}\label{cor2}
Uniformly for $x\ge t\ge 2$, $v=\log x /\log t$,
$$
\frac{T(x,t)}{D(x,t)}=b_t  (\log xt)^{1+\delta} 
\left\{ 1+ O\left(\frac{1}{v^\delta \log t}+\frac{1}{v^{1+\delta}} \right) \right\},
$$
where $b_t=\lambda_0 (1-e^{-\gamma}) \beta_t (\log t)^{-\delta}$, $0<\beta_0\le \beta_t=1+O(1/\log t)$, 
$$
\frac{T(x,t)}{D(x,t)}=\lambda_0 (1-e^{-\gamma}) (v+1)^\delta \log( xt)
  \left\{ 1+ O\left(\frac{1}{\log t}+\frac{1}{v^{1+\delta}} \right) \right\},
$$
$$
\frac{T(x,t)}{D(x,t)}= (\log t) \frac{  \lambda(v)}{d(v)} \left\{ 1+ O\left(\frac{1}{\log t}\right) \right\},
$$
where the function $d(v)$ is as in \cite{PDD}.
\end{corollary}

The $t$-dense integers are a special case of a family of integer sequences that arise as follows. 
Let $\theta$ be an arithmetic function.
Let $\mathcal{B}=\mathcal{B}_\theta$ be the set of positive integers containing $n=1$ and all those $n \ge 2$ with prime factorization $n=p_1^{\alpha_1} \cdots p_k^{\alpha_k}$, $p_1<p_2<\ldots < p_k$, which satisfy 
\begin{equation}\label{Bdef}
p_{i} \le \theta\big( p_1^{\alpha_1} \cdots p_{i-1}^{\alpha_{i-1}} \big) \qquad (1\le i \le k).
\end{equation}
We write $\mathcal{B}(x)=\mathcal{B}\cap [1,x]$ and $B(x)=|\mathcal{B}(x)|$. 
When $\theta(n)=nt$, then $\mathcal{B}$ is the set of $t$-dense integers \cite{Saias1,Ten86,PDD}. 
If $\theta(n)=\sigma(n)+1$, where $\sigma(n)$ is the sum of the positive divisors of $n$, then 
$\mathcal{B}$ is the set of practical numbers \cite{Mar, Saias1, Sier, Sri, Stew, Ten86, PDD}, i.e. integers $n$ such that every $m\le n$ can be expressed as a sum of distinct positive divisors of $n$. By making some adjustments to the proof of Theorem \ref{thm2},
we will obtain the following result. 

\begin{theorem}\label{thmgen}
Assume  $\max(2,n)\le \theta(n) \ll n \exp((\log n)^a)$ for $n\ge 1$ , where $a$ is any constant 
with $a < (1-\delta)/(2-\delta)=0.2226...$ Then
$$
T(x) := \sum_{n\in \mathcal{B}(x)} \tau(n) = \nu_\theta x (\log x)^\delta + O(x),
$$
where $\delta=0.7136125...$ is as in Lemma \ref{LambdaLem} and $\nu_\theta$ is a positive constant.
\end{theorem}

In 1991, Margenstern \cite{Mar} proposed seven conjectures about practical numbers, together with empirical evidence to support them. 
Theorem \ref{thmgen} confirms Conjecture 4 of \cite{Mar}, which says that $T(x)\sim \nu x (\log x)^\delta$ for 
some constants $\nu, \delta \in (1/2,1)$. 
Although we have not determined the constant factor $\nu$, empirical evidence suggests that $\nu =0.54...$ in the case of practical 
numbers. 
Thirty years after the publication of Margenstern's paper \cite{Mar}, six of the seven conjectures have now essentially been resolved \cite{Mel,PW,PDD,CFP, OMB}, while the remaining one \cite[Conjecture 2]{Mar}, which is about the count of twin practical numbers, has seen significant progress \cite{PW,OMB}.

We take this opportunity to give an improved version of \cite[Thm. 5.1]{PDD}, an asymptotic estimate for the 
counting function $B(x)=|\mathcal{B}(x)|$. Theorem \ref{thmimp} has a sharper error term and allows for $\theta(n)$
to grow faster. 
Let
$$l(x)=\exp \left( \frac{\log x}{\log_2 x \log_3^3 x}\right).$$

\begin{theorem}\label{thmimp}
Assume $\max(2,n)\le \theta(n) \le n f(n)$, where $f(x) \ll l(x)$, $f(x)$ is non-decreasing and 
$(\log f(x)) / \log x$ is eventually decreasing.
Then 
$$
B(x) = \frac{c_\theta x}{\log x} \bigl\{1+O\left(E(x) \right)\bigr\},
$$
where
$$
E(x):=\int_x^\infty \frac{\log f(y)}{y\log^2 y} dy
$$
and the positive constant $c_\theta$ is given by
$$
c_\theta = \frac{1}{1-e^{-\gamma}} \sum_{n\in \mathcal{B}} \frac{1}{n}  \Biggl( \sum_{p\le \theta(n)}\frac{\log p}{p-1} - \log n\Biggr) \prod_{p\le \theta(n)} \left(1-\frac{1}{p}\right).
$$
\end{theorem}

For example, if $\theta(n) \ll n (\log n)^A$, where $A$ is an arbitrary constant, then $E(x)\ll \log\log x / \log x$,
while \cite[Thm. 5.1]{PDD} only allows for $A<1$ and has a relative error term of $O(\log^A x / \log x)$ in that case. 
Comparing Theorem \ref{thmimp} with \cite[Thm. 5.1]{PDD}, we have essentially replaced
$f(y)$ with $\log f(y)$ in the definition of the error term $E(x)$. 
The improvement comes from a more careful treatment of the penultimate term in Lemma \ref{PhiLemma},
when this estimate is inserted into Lemma \ref{funceq}.
If $B(x) \sim c_\theta x/\log x $, then $\theta(n)$ cannot grow much faster than is allowed by Theorem \ref{thmimp}, 
 since $\theta(n)\asymp n^{1+\varepsilon}$ implies $B(x)\sim c_\theta x/(\log x)^{1-\eta}$ for some $\eta>0$,
by \cite[Thm. 3]{SPA}.

\section{The function $\xi(u)$}

If we sort the integers $n\le x$ according to their smallest prime factor $p=P^-(n)$ and 
ignore higher powers of $p$, we expect that
$$
S(x,y)=\sum_{n \le x \atop P^-(n)>y} \tau(n) \approx \sum_{\sqrt{x}<p\le x} 2 +  \sum_{y<p\le \sqrt{x}} 2 S(x/p,p).
$$
This suggests the definition \eqref{xidef}, provided $S(x,y) \approx x \xi(u)/\log y$. 
We have 
$$
u \xi(u)=
\begin{cases}
0 & u<1 \\
2 & 1\leq u<2 \\
4 \log (u-1)+2  \quad  & 2\leq u< 3 
\end{cases}
$$
and
\begin{equation}\label{xidiffeq}
(u \xi(u))'= 2 \xi(u-1) \qquad (u\neq 1,2).
\end{equation}
As with Buchstab's function $\omega(u)$, $\xi(u)$ is $k$ times continuously differentiable for $u>k+1$. 
By induction on $\lfloor u \rfloor$, \eqref{xidef} easily implies the bounds 
\begin{equation}\label{xiexbounds}
\frac{u+2}{4} \le \xi(u) \le u+1 \qquad (u\ge 1).
\end{equation}

Equation \eqref{xidiffeq} shows that $\xi(u)$ belongs to a family of functions $f$ that satisfy 
$uf'(u) + a f(u) + bf(u-1) = 0$, where $a$ and $b$ are constants. These equations were studied in great detail 
by  Hildebrand and Tenenbaum \cite{DDE}.

\begin{lemma}\label{xilemma}
We have
\begin{align}
\xi(u)  =2\omega(u) & + \int_0^u \omega(t) \omega(u-t) dt \qquad & (u\ge 0), \label{xilemeq1} \\
|\xi''(u)| & < \frac{2^u}{7\Gamma(u+1)} \qquad &  (u\ge 3.5) , \label{xilemeq2} \\
|\xi'(u)-e^{-2\gamma}| & < \frac{2^u}{7\Gamma(u+1)} \qquad & (u\ge 2.5),  \label{xilemeq3} \\
|\xi(u)-(u+2)e^{-2\gamma}| & < \frac{2^u}{7\Gamma(u+1)} \qquad & (u\ge 1.5), \label{xilemeq4} \\
\xi(u)- (u+2)e^{-2\gamma} & \ll  \left(\frac{e}{u\log 2u}\right)^{u} \qquad & (u\ge 1). \label{xilemeq5} 
\end{align}
\end{lemma}
\begin{proof}
With equation \eqref{xidiffeq}, we find that the Laplace transform 
$\hat{\xi}(s)$ is given by
\begin{equation}\label{xilap}
1+\hat{\xi}(s) = s^{-2} \exp\left( -2\gamma +2 \int_0^s (1-e^{-t}) t^{-1} dt \right) = (1+\hat{\omega}(s))^2,
\end{equation}
by \cite[Thm. III.6.5]{Ten}. This implies \eqref{xilemeq1}. 

For the next three estimates, we find exact formulas for $\xi(u), \xi'(u), \xi''(u)$ in terms of polylogarithms, on
the interval $[0,5]$, with the help of Mathematica. This shows that the claims are valid for $u\le 5$. 
For $u>3$, $\xi(u)$ is twice continuously differentiable, so that \eqref{xidiffeq} implies
$$
u \xi''(u) = 2\xi'(u-1)-2\xi'(u) = - 2 \int_{u-1}^u \xi''(t) dt \qquad (u>4).
$$
Assuming that \eqref{xilemeq2} is false, we define
$$
M:=\inf\{ u\ge 3.5: |\xi''(u)| \ge 2^u/(7\Gamma(u+1)) \},
$$
so that $M>5$. We have
$$
M |\xi''(M)| \le  2 \int_{M-1}^M |\xi''(t)| dt < 2 \int_{M-1}^M \frac{ 2^t}{7\Gamma(t+1)} dt < 2 \frac{2^{M-1}}{7\Gamma(M)},
$$
since $ 2^t/(7\Gamma(t+1))$ is decreasing for $t\ge 2$. Dividing by $M$, we get $ |\xi''(M)| < 2^M/(7\Gamma(M+1)$, which 
is impossible in light of the continuity of $\xi''(u)$ at $u=M>5$. This establishes \eqref{xilemeq2}.

Since $\omega(u)-e^{-\gamma}, \omega'(u) \ll 1/\Gamma(u+1)$ and $\int_0^\infty (\omega(u)-e^{-\gamma})du = e^{-\gamma}-1$,
\eqref{xilemeq1} implies $\xi(u) - (u+2)e^{-2\gamma}=o(1)$ and $\xi'(u)-e^{-2\gamma} = o(1)$ as $u\to \infty$. 
The estimate \eqref{xilemeq2} yields
$$
\left|\xi'(u)-e^{-2\gamma} \right| = \left| -\int_u^\infty \xi''(t) dt \right| \le \int_u^\infty \frac{ 2^t}{7\Gamma(t+1)} dt 
< \frac{ 2^u}{7\Gamma(u+1)},
$$
where the last inequality is valid for $u\ge 5$. Similarly, \eqref{xilemeq3} yields
$$
\left|\xi(u)-(u+2)e^{-2\gamma} \right| = \left| \int_u^\infty (\xi'(t)-e^{-2\gamma}) dt \right| \le \int_u^\infty \frac{ 2^t dt}{7\Gamma(t+1)} 
< \frac{ 2^u}{7\Gamma(u+1)},
$$
for $u\ge 5$.

The estimate \eqref{xilemeq5} follows from \cite[Thms. 1 and 2 ]{DDE}, which contains more
precise information about the error term.
\end{proof}

\section{Proof of Theorem \ref{thm1}}\label{SecProofThm1}

Define
$$
\Phi(x,y) = |\{n\le x: P^-(n)>y\}|.
$$
In Lemma 5 of \cite{OMB}, we obtained the following estimate for $\Phi(x,y)$
by combining several results from Section III.6 of Tenenbaum \cite{Ten}. 
Having three separate terms of size $x/\log y$ in our estimate of $\Phi(x,y)$, 
rather than one composite term, simplifies our proof.

\begin{lemma}\label{PhiLemma}
Uniformly for $x\ge 1$, $y\ge 2$, $u=\log x /\log y$,
\begin{equation*}
\begin{split}
\Phi(x,y) &= 1+ x \prod_{p\le y}\left(1-\frac{1}{p}\right)
+\frac{x}{\log y}\! \left\{\omega(u)-e^{-\gamma}\! -\left.\frac{y}{x}\right|_{x\ge y} 
+ O\left(\frac{e^{-u/3}}{\log y}\right)\! \right\}.\\
\end{split}
\end{equation*}
\end{lemma}

\begin{lemma}\label{Rlem}
Uniformly for $x\ge 1$, $y\ge 2$, $u=\log x /\log y$,
$$
\sum_{n\le x \atop P^-(n)>y} \frac{1}{n} 
=1+  (\log x) \prod_{p\le y}\left(1-\frac{1}{p}\right)
+\int_0^u (\omega(v)-e^{-\gamma})dv + O\left(\frac{1}{\log y}\right).
$$
\end{lemma}
\begin{proof}
This follows from Lemma \ref{PhiLemma} and partial summation.
\end{proof}

\begin{proof}[Proof of Theorem \ref{thm1}]
We write
$$
S(x,y):= \sum_{n\le x \atop P^-(n)>y} \tau(n) =  \sum_{n\le x \atop P^-(n)>y}  \sum_{d|n} 1 
=  \sum_{d\le x \atop P^-(d)>y} \Phi(x/d,y).
$$
With Lemma \ref{PhiLemma}, we obtain
\begin{equation*}
\begin{split}
S(x,y)  = &  O\left(\frac{x}{\log^2 y}\right) + \Phi(x,y) 
+ \Biggl( \sum_{d\le x \atop P^-(d)>y} \frac{1}{d}\Biggr)  x \prod_{p\le y}\left(1-\frac{1}{p}\right) \\
& + \frac{x}{\log y}  \sum_{d\le x \atop P^-(d)>y} \frac{1}{d}
\left\{ \omega\left(\frac{\log x/d}{\log y}\right)-e^{-\gamma} -\left.\frac{y}{x/d}\right|_{d\le x/y}\right\}.
\end{split}
\end{equation*}
We estimate $\Phi(x,y)$ with Lemma \ref{PhiLemma}, and the first sum with Lemma \ref{Rlem}. 
In the second sum, the contribution from $d=1$ is 
$$
\frac{x(\omega(u)-e^{-\gamma}) -y|_{x\ge y} }{\log y},
$$
while the contribution from $1<d\le x$ is 
$$
\frac{x}{\log y} \int_0^u \omega(v)(\omega(u-v)-e^{-\gamma}) dv +  O\left(\frac{x}{\log^2 y}\right),
$$
by partial summation, Lemma  \ref{Rlem} and integration by parts. 
Combining these estimates with Mertens' formula, we obtain Theorem \ref{thm1}.
\end{proof}

\section{Proof of Theorem \ref{thm2}}

We assume that $\theta$ satisfies
\begin{equation}\label{thetadef}
\theta : \mathbb{N} \to \mathbb{R}\cup \{\infty\}, \quad \theta(1)\ge 2, \quad \theta(n)\ge P^+(n) \quad (n\ge 2).
\end{equation}

\begin{lemma}\label{funceq}
Assume \eqref{thetadef} and let $f(n)$ be multiplicative. For $x\ge 0$ we have
$$
\sum_{m\le x} f(m) = \sum_{n\in \mathcal{B}(x)} f(n) 
\Bigl( 1 + \sum_{2\le r\le x/n \atop P^-(r)>\theta(n)} f(r)\Bigr).
$$
\end{lemma}
\begin{proof}
Each $m \ge 1$ factors uniquely as $m=n r$, with $n\in \mathcal{B}$ and 
$P^-(r)>\theta(n)$ if $r>1$. 
\end{proof}

\begin{theorem}\label{thmL}
Assume \eqref{thetadef}. 
 We have
$$
\sum_{n\in\mathcal{B}(x)} \tau(n) = (1-L) x\log x + o(x \log x),
$$
where $0\le L \le 1$ and 
\begin{equation}\label{Ldef}
L= \sum_{n\in \mathcal{B}} \frac{\tau(n)}{n} \prod_{p\le \theta(n)} \left(1-\frac{1}{p}\right)^2.
\end{equation}
Thus, $L=1$ if and only if $\sum_{n\in\mathcal{B}(x)} \tau(n) =o(x\log x)$.
\end{theorem}

\begin{proof}
For every $N\le x$, Lemma \ref{funceq} yields
$$
\sum_{m\le x} \tau(m) \ge \sum_{n\in \mathcal{B}(N)} \tau(n) 
\Bigl( 1 + \sum_{2\le r\le x/n \atop P^-(r)>\theta(n)} \tau(r)\Bigr).
$$
We divide both sides by $x\log x$, fix $N$ and let $x \to \infty$. Theorem \ref{thm1} implies
$$
1 \ge \sum_{n\in \mathcal{B}(N)} \frac{\tau(n)}{n} \prod_{p\le \theta(n)} \left(1-\frac{1}{p}\right)^2.
$$
This shows that the series in \eqref{Ldef} converges to a value $L \in [0,1]$. 
Inserting Theorem \ref{thm1} into the last sum of Lemma \ref{funceq}, the convergence of the series \eqref{Ldef} implies
$$
x \log x = \sum_{n\in\mathcal{B}(x)} \tau(n) + L x \log x + o(x\log x).
$$
\end{proof}

\begin{corollary}\label{L1cor}
Assume \eqref{thetadef}. If $\max(2,n)\le \theta(n) \ll n^{2}$ for $n\ge 1$, then 
$$
x (\log x)^{0.58} \ll 
\sum_{n\in \mathcal{B}(x)} \tau(n) 
\ll  x (\log x)^{0.99998}
$$
and $L=1$. 
\end{corollary}
\begin{proof}
Theorem 3 and Table 1 of \cite{SPA} show that $\theta(n)\ll n^{2}$ implies $B(x) \ll x (\log x)^{-0.4191}.$ 
We have 
$$
\sum_{n\in \mathcal{B}(x) \atop \Omega(n) \le 2.01 \log_2 x} \tau(n) \le B(x)  2^{2.01 \log_2 x} 
\ll \frac{ x(\log x)^{(2.01) \log 2}}{ (\log x)^{0.4191}} \ll x (\log x)^{0.98} 
$$
and
$$
\sum_{n\in \mathcal{B}(x) \atop \Omega(n) \ge 2.01 \log_2 x} \tau(n) 
\le \sum_{n\le x \atop \Omega(n) \ge 2.01 \log_2 x} \tau(n) 
\ll  x (\log x)^{0.99998},
$$
by Lemma \ref{lemOmTau} with $\alpha = 2.01$. Theorem \ref{thmL} shows that $L=1$. 

The lower bound follows from \cite[Cor. 3]{OMB}. The exponent $0.58$ can be replaced by any constant less than 
$C \log 2 -1$, where $C=(1-e^{-\gamma})^{-1}$.
\end{proof}
\begin{lemma}[Lemma 2 of \cite{OMB}]\label{lemOmTau}
Let $\varepsilon >0$. For $2\le \alpha \le 4-\varepsilon$ we have
$$
\sum_{n\le x \atop \Omega(n) \ge \alpha \log_2 x} \tau(n) \ll x (\log x)^{\alpha(\log 2 - \log \alpha +1)-1}.
$$
\end{lemma}

\begin{proof}[Proof of Theorem \ref{thm2}]
We insert Theorem \ref{thm1} into the last sum of Lemma \ref{funceq}.
The contribution from the error term of Theorem \ref{thm1} is $O(x)$, since $L\le 1$ by Theorem \ref{thmL}.
The contribution form the penultimate term of Theorem \ref{thm1} is 
$$
\ll \sum_{n\in \mathcal{D}(x,t)\atop n \le \sqrt{x/t}} \tau(n) \frac{nt}{\log nt} 
\ll \frac{\sqrt{xt}}{\log \sqrt{xt}} \sum_{n\le \sqrt{x/t}} \tau(n) \ll x.
$$
For the remaining terms, we use the fact that $L=1$ for each $t\ge 2$, by Corollary \ref{L1cor}, and
that $\xi(u)=0$ for $u<0$. With Mertens' formula, we find that, for $x\ge 1$, 
$$
T(x,t)=x\sum_{n\in \mathcal{D}_t} \frac{\tau(n)}{n\log nt} 
\left\{\frac{\log x}{\log nt} e^{-2\gamma}-\xi\left(\frac{\log x/n}{\log nt}\right)\right\} +O(x),
$$
where $\mathcal{D}_t$ denotes the set of all $t$-dense integers.
Since $\xi(u)=(u+ 2)e^{-2\gamma} + O(e^{-u})$ and $\xi'(u)=e^{-2\gamma} + O(e^{-u})$ for $u>2$,
partial summation yields
\begin{equation*}
T(x,t)=
 x \log x \int_1^\infty \frac{T(y,t)e^{-2\gamma}}{y^2 \log^2 yt}dy
- x\int_1^x \frac{T(y,t)}{y^2 \log yt} \, \xi\left(\frac{\log xt}{\log yt} -1\right)dy +O(x),
\end{equation*}
for $x\ge 1$. 
With the change of variables
$$
x=t^{e^z-1}, \quad y=t^{e^u-1}, \quad G_t(z) = \frac{T(x,t)}{x},
$$
we obtain, for $z\ge 0$,
\begin{equation*}
G_t(z)=e^{-2\gamma} \hat{G}_t(1) (e^z-1) - \int_0^z G_t(u) h(z-u) du   + E_t(z),
\end{equation*}
where  $h(z)=\xi(e^{z}-1)$ and  $E_t(z)=O(1)$. 
We multiply by $e^{-sz}$, where $\re s >1$, and integrate over $z\ge 0$, to get the equation of Laplace transforms,
$$
\hat{G}_t(s)=\frac{e^{-2\gamma} \hat{G}_t(1)}{s(s-1)}- \hat{G}_t(s) \hat{h}(s)
+\hat{E}_t(s) \qquad (\re s>1).
$$
Solving for $\hat{G}_t(s)$, we find that
$$
\hat{G}_t(s)=e^{-2\gamma} \hat{G}_t(1) \hat{\Lambda}(s)+ \hat{\Lambda}(s) s(s-1) \hat{E}_t(s)\qquad (\re s>1),
$$
where $\hat{\Lambda}(s)$ is as in Lemma \ref{LambdaLem}. Since $\Lambda(0)=0$ and $\Lambda'(0)=1$,
$$
G_t(z) = e^{-2\gamma} \hat{G}_t(1) \Lambda(z) + \int_0^z \Lambda''(z-u) E_t(u)du
-\int_0^z \Lambda'(z-u) E_t(u)du +E_t(z),
$$
for $z\ge 0$. 
Now $E_t(z) = O(1)$, so Lemma \ref{LambdaLem} yields
\begin{equation*}
\begin{split}
\int_0^z \Lambda'(z-u) E_t(u)du  & = \int_0^z \lambda_0 \delta e^{\delta (z-u)} E_t(u)du +O(1) \\
& = \lambda_0 \delta e^{\delta z} \int_0^\infty  e^{-\delta u} E_t(u)du +O(1) \\
&  = \lambda_0 \delta \hat{E}_t(\delta)  e^{\delta z}+O(1)\\
& = \delta \hat{E}_t(\delta)  \Lambda(z) +O(1).
\end{split}
\end{equation*}
Similarly,
$$
\int_0^z \Lambda''(z-u) E_t(u)du 
 = \delta^2 \hat{E}_t(\delta)  \Lambda(z) +O(1).
$$
Since $L=1$, $ e^{-2\gamma} \hat{G}_t(1)=\log t +O(1)$, by partial summation. Thus,
$$
G_t(z) = \alpha_t (\log t) \Lambda(z) +O(1),
$$
where 
$$
\alpha_t = \frac{ e^{-2\gamma} \hat{G}_t(1) + \delta(\delta-1) \hat{E}_t(\delta) }{\log t} 
= 1+O\left(\frac{1}{\log t}\right).
$$
The claim $\alpha_t\ge \alpha_0>0$ follows from the lower bound in Corollary \ref{L1cor}.
Together with Lemma \ref{LambdaLem}, this completes the proof of Theorem \ref{thm2}.
\end{proof}

\begin{lemma}\label{LambdaLem}
Let $\Lambda(z) = \lambda(e^z-1)$ and $h(z)=\xi(e^{z}-1)$. 
The Laplace transform of $\Lambda$ is given by
\begin{equation}\label{Lambdahat}
  \hat{\Lambda}(s)= \frac{1}{s(s-1)(1+\hat{h}(s))},
  \end{equation}
$\hat{\Lambda}(s)$ has a simple pole at $s=\delta= 0.7136125...$ with residue 
$\lambda_0=1.118192...$, a simple pole at $s=-1$ with residue 
 $$
 \lambda_1 = \frac{2}{3 e^{-2\gamma}-2} = -1.897011...,
 $$
and a pair of simple poles at $-1.962...\pm 11.57...i$ with corresponding residues $-0.0078...\pm 0.0031... i$.
There are no other singularities with $\re s \ge -3$. 
We have
$$
\Lambda(z) = \lambda_0 e^{\delta z} + \lambda_1 e^{-z} + O(e^{-1.962z}),
 $$ 
 $$
\Lambda'(z) = \lambda_0 \delta e^{\delta z} - \lambda_1 e^{-z} + O(e^{-1.962z}),
 $$ 
  $$
\Lambda''(z) = \lambda_0 \delta^2 \delta e^{\delta z} + \lambda_1 e^{-z} + O(e^{-1.962z}).
 $$ 
The constant $\delta$ is the unique value of $s\in (0,1)$ such that 
$$
0=1+\int_0^\infty \left(\xi(v)-(v+2)e^{-2\gamma}\right) \frac{dv}{(v+1)^{s+1}}
+\frac{e^{-2\gamma}}{s}+\frac{e^{-2\gamma}}{s-1}.
$$
We have $\lambda_0=(\delta (\delta-1)I)^{-1}$, where 
 $$
  I=-\int_0^\infty\left(\xi(v)-(v+2)e^{-2\gamma}\right) 
 \frac{\log(v+1)}{(v+1)^{1+\delta}}dv
 -\frac{e^{-2\gamma} }{\delta^2}-\frac{e^{-2\gamma} }{(\delta-1)^2}.
 $$
\end{lemma}

\begin{proof}
Since $0\le \xi(u)\le u+1$ for $u \ge 0$, by \eqref{xiexbounds}, the definition \eqref{lambdadef} of $\lambda(v)$ 
yields, by induction on
$\lfloor v\rfloor$, the preliminary (and very crude) estimate $|\lambda(v)| \le (v+1)^2$, and hence $|\Lambda(z)| \le e^{2z}$. 
It follows that the Laplace transform $\hat{\Lambda}(s)=\int_0^\infty \Lambda(z) e^{-zs} dz $ converges 
absolutely for $\re s >2$ and is analytic there. 
The relation \eqref{Lambdahat} follows from \eqref{lambdadef}. We write
$$ g(s):=1+\hat{h}(s)=  1+\int_0^\infty \left(\xi(v)-(v+2)e^{-2\gamma}\right) \frac{dv}{(v+1)^{s+1}}
+\frac{e^{-2\gamma}}{s}+\frac{e^{-2\gamma}}{s-1}.
$$
Since $| \xi(v)-(v+2)e^{-2\gamma}| \ll v^{-v}$, by Lemma \ref{xilemma}, the integral is an entire function of $s$,
which means that $g(s)$ is analytic everywhere except for simple poles at $s=0$ and $s=1$.
Thus, \eqref{Lambdahat} extends $\hat{\Lambda}(s)$ to a meromorphic function on the entire complex plane, whose poles are exactly the zeros of $g(s)$.

Since $\xi(v)=0$ for $v<1$, integration by parts shows that
\begin{equation}\label{gbyparts}
g(s) = 1 +2^{1-s} \left( \frac{e^{-2\gamma}}{s-1} + \frac{1-e^{-2\gamma}}{s}\right) +\frac{1}{s} \int_1^\infty \left(\xi'(v)- e^{-2\gamma}\right) \frac{dv}{(v+1)^s}.
\end{equation}
Thus, if $g(s)=0$, we must have
$$
|\tau|\le  H(\sigma) := 2^{1-\sigma} + \int_1^\infty \left|\xi'(v)- e^{-2\gamma}\right| \frac{dv}{(v+1)^\sigma}.
$$

Let $P_a$ denote the finite set of poles of $\hat{\Lambda}(s)$ (i.e. zeros of $g(s)$) with $ \re s \ge -a$. 
Since $\hat{\Lambda}(s)$ is analytic for $\sigma>2$ and $H(\sigma)$ is decreasing, $P_a$ is a subset of the 
finite rectangle $-a\le\sigma \le2$, $|\tau| \le H(-a)$. 

 From \eqref{lambdadef}, $\lambda(v)$ is continuously differentiable for $v>0$,
 $\lambda'(v)$  is continuously differentiable for $v>0, v\neq 1$ and $\lambda''(v)$ has 
 a finite jump discontinuity at $v=1$.
 Let $F(z)$ be any one of $\Lambda(z), \Lambda'(z), \Lambda''(z)$, and $\hat{F}(s)$ be its Laplace transform.
We have
$$F(z) = \frac{1}{2\pi i} \int_{3-i\infty}^{3+i\infty} \hat{F}(s) e^{zs}\, d s.$$
Let $T=e^{z(a+3)}$. Since the result is trivial for bounded $z$, we may assume that $z$ is sufficiently large such that $T>2H(-a)$.
We have $ |\hat{h}(s)| \le \frac{H(\sigma)}{|\tau|}$, so that $ |\hat{h}(s)| \le 1/2 $ whenever $|\tau|\ge 2 H(\sigma)$.
Thus,
$$
\frac{1}{1+\hat{h}(s)}=1-\hat{h}(s) + O(\hat{h}(s)^2) 
= 1-2^{1-s} \left( \frac{e^{-2\gamma}}{s-1} + \frac{1-e^{-2\gamma}}{s}\right) +O_\sigma(\tau^{-2}),
 $$ 
 for $|\tau|\ge 2 H(\sigma)$, since the integral in \eqref{gbyparts} is $O_\sigma(\tau^{-1})$, as can be seen by
 applying integration by parts a second time. With \eqref{Lambdahat}, it follows that the Laplace transforms of $\Lambda(z), \Lambda'(z)$ and 
 $\Lambda''(z)$ satisfy
 $$
 \hat{\Lambda}(s)\ll_a \tau^{-2}, \qquad  \widehat{\Lambda'}(s)=s  \hat{\Lambda}(s)= \frac{1}{s-1} +O_a(\tau^{-2}),
$$
$$
\widehat{\Lambda''}(s)
= s(s-1) \hat{\Lambda}(s) +s  \hat{\Lambda}(s)-1
= \frac{1}{s-1} -2^{1-s} \left( \frac{e^{-2\gamma}}{s-1} + \frac{1-e^{-2\gamma}}{s}\right) +O_a(\tau^{-2}),
$$
for $|\tau|\ge  2 H(-a)$ and $-a\le \sigma \le 2$.
We find that
$$ \int_{3+iT}^{3+i\infty}  \hat{F}(s) e^{zs}\, d s =O\left(e^{-az}\right),$$
$$
 \int_{-a+iT}^{3+iT}  \hat{F}(s) e^{zs}\, d s \ll_a \int_{-a}^3 \frac{e^{z\sigma}}{T} d\sigma = O_a\left(e^{-az}\right),
$$
and
$$
 \int_{-a+i2H(-a)}^{-a+iT}  \hat{F}(s) e^{zs}\, d s = O_a\left(e^{-az}\right).
$$
For the remaining segment, we assume that $\hat{\Lambda}(s)$ has no poles with $\sigma =a$. 
If it does, one just replaces $-a$ by $-a-\varepsilon$. We have
$$
 \left| \int_{-a-i2H(-a)}^{-a+i2H(-a)} \hat{F}(s) e^{zs}\, d s \right| 
 \le 4H(-a) \max_{|\tau|\le 2H(-a)} \left|\hat{F}(-a+i\tau)\right| e^{-az} = O_a\left(e^{-az}\right).
 $$
The residue theorem now yields
$$ F(z)= \sum_{s_k \in P_a} \res\left( \hat{F}(s) e^{zs}; s_k \right)+O_a\left(e^{-az}\right).$$

We now determine the poles with $\sigma \ge -3$. 
To estimate $H(-3)$, we use exact values of $\xi'(v)$ on $[1,5]$, and estimate the tail 
with the help of \eqref{xilemeq3}. 
This shows that for $\sigma \ge -3$, all zeros of $g(s)$ satisfy $|\tau| \le H(-3) < 62$. 

For real $s \in (0,1)$, we have $\lim_{s\to 0^+}g(s)=+\infty$ and $\lim_{s\to 1^-}g(s)=-\infty$, so 
$g(s)$ has at least one real zero in the interval $(0,1)$. We will see in a moment that there is exactly one such zero,
say at $s=\delta \in (0,1)$. We have $g(-1)=0$, since  \eqref{xidiffeq} implies
\begin{equation}\label{Req}
\frac{d}{dv} (v+1) \left( \xi(v+1)-(v+3) e^{-2\gamma}\right) =2\left(\xi(v)-(v+2) e^{-2\gamma}\right),
\end{equation}
for $v>0$, $v \neq 1$.
Define
$$ g_5(s)=1+\int_0^5 \left(\xi(v)-(v+2)e^{-2\gamma}\right) \frac{dv}{(v+1)^{s+1}}
+\frac{e^{-2\gamma}}{s}+\frac{e^{-2\gamma}}{s-1}.
$$
By using exact formulas for $\xi(v)$ on the interval $[0,5]$, we can calculate $g_5(s)$ with arbitrary precision.

Let $R$ be the rectangle defined by the lines $\text{Re}(s)=\pm 3$, $\text{Im}(s)=\pm 62$.
We evaluate the contour integral
$$ \frac{1}{2\pi i}\int_R \frac{g_5'(s)}{g_5(s)} ds=2,$$
by numerical integration. This shows that $g_5(s)$ has exactly four zeros (and two poles) inside of $R$.
We need to estimate the error 
$$|g(s)-g_5(s)| \le \int_5^\infty\left|\xi(v)-(v+2)e^{-2\gamma}\right|\, \frac{d v}{(v+1)^{\sigma+1} }.
$$ 
Lemma \ref{xilemma} shows that 
$$
\xi(v)-(v+2)e^{-2\gamma}=2r(v)+2e^{-\gamma}(v+1)r(v+1)+\int_0^v r(u)r(v-u) du,
$$
where $r(v)=\omega(v)-e^{-\gamma}$. 
With  a table of zeros and relative extrema of $\omega(u)-e^{-\gamma}$ on the interval $[5, 10.3355]$ due to Cheer and Goldston \cite{CG}, and the estimate $|\omega(u)-e^{-\gamma}|<1/\Gamma(u+1)$ from \cite[Lemma 1]{IDD3}, 
we find that 
$$ |g(s)-g_5(s)| < 0.0035 \quad (\text{Re}(s)\ge -3). $$
On the boundary of the rectangle $R$, we find that $|g_5(s)|>0.0051 $ .
Rouch\'e's Theorem now shows that $g(s)$ also has exactly four zeros in the rectangle $R$. 
Applying a similar argument to the small square $Q$, bounded by the lines $\text{Re}(s)=-1.963, -1.961$, $\text{Im}(s)=11.574, 11.576$, we find that $g(s)$ has a zero inside $Q$. 
Similarly, we find that the real zero at $\delta \in (0,1)$ satisfies $0.713611 < \delta < 0.713614$. 
Thus, the four zeros of $g(s)$ with $\text{Re}(s)\ge -3$ are $\delta=0.71361...$, $-1$,
$-1.96... \pm 11.57...i$.
Replacing $g_5(s)$ by $g_6(s)$, where we use numerical integration to evaluate $\xi(v)$ for $5<v\le 6$, we find that
$\delta=0.7136125...$ and the complex zeros are at $-1.962... \pm 11.57... i$. 

Since all poles $s_k$ with $\re s_k \ge -3$ are simple,  the corresponding residues satisfy 
$\lambda_k = (s_k (s_k-1) g'(s_k))^{-1}$.  
In the case of $\lambda_1$, integration by parts and \eqref{Req} allow us to find the exact value.
\end{proof}

\section{Another equation for $\delta$}

Besides the definition of  $\delta$ in Lemma \ref{LambdaLem}, $\delta$ satisfies another equation, which
allows us to independently confirm the numerical calculations of Lemma \ref{LambdaLem}. 
Let
$$
Q(s):=\int_0^\infty u^s \left(e^{2J(u)}-1\right) du \qquad (\re s >1),
$$
where 
$$
J(u):=  \int_u^\infty \frac{e^{-t}}{t } dt.
$$
If we develop $e^{2J(u)}-1$ into the series $\sum_{k\ge 0} b_k u^{k-2}$, and write
$$e^{2 J(u)}-1=\left(e^{2 J(u)}-1-\sum_{k=0}^K b_k u^{k-2}\right) + \sum_{k=0}^K b_k u^{k-2},$$ 
we find that the contribution from $u\in [0,1]$ to $Q(s)$, and hence $Q(s)$ itself, 
extends to a meromorphic function on all of $\mathbb{C}$, with a simple pole at every $s=1-k$, 
$k\in \mathbb{N}\cup \{0\}$, for which $b_k\neq 0$. 

From the proof of Lemma \ref{LambdaLem}, we know that $\hat{h}(s)$, the Laplace transform of $h(z)=\xi(e^z-1)$, 
extends to a meromorphic function
on all of $\mathbb{C}$, whose only poles are at $s=0,1$. Moreover, the poles of $\hat{\Lambda}(s)$ are 
precisely the zeros of $1+ \hat{h}(s)$. With Lemma \ref{altdef}, we can locate these zeros by 
studying $Q(s)$, without the need to evaluate $\xi(v)$.  

\begin{lemma}\label{altdef}
For $s \in \mathbb{C}$ we have
$$(s+1) Q(s)=2\Gamma(s+1) \left(1+\hat{h}(s)\right).$$
\end{lemma}

\begin{proof}
For $\re s > 1$, equation \eqref{xilap} implies
\begin{equation*}
\begin{split}
Q(s) 
& =  \int_0^\infty u^s \hat{\xi}(u) du 
 =  \int_0^\infty u^s \int_0^\infty \xi(v) e^{-uv} dv \, du \\
& =  \int_0^\infty \xi(v)  \int_0^\infty u^s e^{-uv} du \, dv \\
& =  \int_0^\infty \xi(v)  \int_0^\infty \left(\frac{w}{v}\right)^s e^{-w} \frac{dw}{v} \, dv \\
& = \Gamma(s+1) \int_0^\infty v\xi(v) \frac{dv}{v^{s+2}}\\
&= \frac{2\Gamma(s+1)}{s+1}\left( 1+ \int_0^\infty \xi(v) \frac{dv}{(v+1)^{s+1}}\right),
\end{split}
\end{equation*}
where the last equation follows from \eqref{xidiffeq} and integration by parts.
\end{proof}

It follows from Lemma \ref{altdef} that the constant $\delta$ in Lemma \ref{LambdaLem} and Theorem \ref{thm2}  
is the unique value of $s\in (0,1)$ which satisfies $Q(s)=0$. 

\begin{remark}
Let $\Omega(z)=\omega(e^z-1)$. 
The poles of the Laplace transform appearing in \cite[Lemma 5]{IDD3}, 
in connection with the study of $D(x,t)$, are the zeros of $(1+\hat{\Omega}(s-1))$.
The same reasoning as in Lemma \ref{altdef} shows that 
$$
s\int_0^\infty u^{s-1} \left(e^{J(u)}-1\right) du =\Gamma(s) \left(1+\hat{\Omega}(s-1)\right).
$$
This provides an alternative for locating the poles of the Laplace transform in \cite[Lemma 5]{IDD3}, without the 
need to approximate $\omega(u)$. 
\end{remark}

\section{Proof of Theorem \ref{thmgen}}

\begin{lemma}\label{lemfinsum}
Assume  $\max(2,n)\le \theta(n) \ll n \exp(\log n)^a)$ for $n\ge 1$ , where $a$ is any constant with 
$a < (1-\delta)/(2-\delta)=0.2226...$ Then $L=1$,
$$
T(x) := \sum_{n \in \mathcal{B}(x)} \tau(n) \ll x (\log x)^{\delta + a(1-\delta)}
 $$ 
 and
$$
\sum_{n\in \mathcal{B}} \frac{\tau(n)\log^a n}{n\log^2 n} \ll 1.
$$
\end{lemma}

\begin{proof}
Since $\theta(n)\le A n \exp((\log n)^a)$ for some constant $A$, $\mathcal{B}(x) \subset \mathcal{D}(x, t)$, where
$t= A  \exp((\log x)^a)$. Corollary \ref{ordmag} shows that 
$$
\sum_{n \in \mathcal{B}(x)} \tau(n) \le  \sum_{n \in \mathcal{D}(x,t)} \tau(n) \ll x (\log t) v^\delta
\ll x (\log x)^{\delta + a(1-\delta)}.
$$
The first claim now follows from Theorem \ref{thmL}. Partial summation yields the last claim. 
\end{proof}

\begin{proof}[Proof of Theorem \ref{thmgen}]
We follow the proof of Theorem \ref{thm2}, with some adjustments. 
We insert Theorem \ref{thm1} into the last sum of Lemma \ref{funceq}.
As before, the contribution from the error term of Theorem \ref{thm1} is $O(x)$, since $L\le 1$ by Theorem \ref{thmL}.
To estimate the contribution form the penultimate term of Theorem \ref{thm1}, 
note that $\log \theta(n) \asymp \log (n \theta(n))$. We write 
\begin{equation}\label{firsteq}
\sum_{n\in \mathcal{B}(x)\atop n \theta(n) \le x} \tau(n) \frac{\theta(n)}{\log \theta(n)} 
\asymp \sum_{k\ge 0} \frac{x}{2^k \log (x/2^k)}\sum_{n\in \mathcal{B}(x)\atop n \theta(n) \in I_k } \frac{\tau(n)}{n},
\end{equation}
where $I_k=(x/2^{k+1},x/2^k]$. The contribution from $k$ with $2^k>\sqrt{x}$ is trivially $O(x)$, since
$\sum_{n\le x} \tau(n)/n \ll \log^2 x$. If $2^k \le \sqrt{x}$, then $\log(x/2^k) \asymp \log x$. 
The conditions on $\theta$ mean that $n\theta(n) \in I_k$ implies that 
$n\in (M_k, N_k]$ where $N_k=\sqrt{x 2^{-k}}$ and $\log (N_k/M_k) \ll (\log x)^a$. 
With $b:=\delta+a(1-\delta)$, partial summation and Lemma \ref{lemfinsum} yield
$$
\sum_{n\in \mathcal{B}(x)\atop n \theta(n) \in I_k } \frac{\tau(n)}{n}
\ll \frac{T(N_k)}{N_k} + \int_{M_k}^{N_k} \frac{T(y)}{y^2} dy 
\ll (\log N_k)^{b}+ \frac{(\log N_k)^{b+1}-(\log M_k)^{b+1}}{b+1}.
$$
Writing $\log M_k = \log N_k - \log N_k/M_k = \log N_k +O(\log^a x)$, we find that the last expression 
is $\ll (\log x)^{a+b} \le \log x$. Inserting this into \eqref{firsteq}, we conclude that the 
contribution from the penultimate term of Theorem \ref{thm1} is $O(x)$. 

For the remaining terms, we use the fact that $L=1$, by Lemma \ref{lemfinsum}, and
that $\xi(u)=0$ for $u<0$. With Mertens' formula, we find that, for $x\ge 1$, 
$$
T(x)=x\sum_{n\in \mathcal{B}} \frac{\tau(n)}{n\log \theta(n)} 
\left\{\frac{\log x}{\log \theta(n)} e^{-2\gamma}-\xi\left(\frac{\log x/n}{\log \theta(n)}\right)\right\} +O(x).
$$
We now replace each occurrence of $\theta(n)$ by $2n$. Since $\xi(u)=(u+2)e^{-2\gamma}+O(e^{-u})$ and $\xi'(u)=e^{-2\gamma}+O(e^{-u})$, Lemma \ref{lemfinsum} shows that 
$$
T(x)=x\sum_{n\in \mathcal{B}} \frac{\tau(n)}{n\log 2n} 
\left\{\frac{\log x}{\log 2n} e^{-2\gamma}-\xi\left(\frac{\log x/n}{\log 2n}\right)\right\} +O(x).
$$
The rest of the proof is identical to that of Theorem \ref{thm2}, with $t=2$. 
The conclusion is that $T(x)= x \alpha_2 (\log 2) \lambda(v)+O(x)$.  
Theorem \ref{thmgen} now follows from \eqref{lamapp}.
\end{proof}

\section{Proof of Theorem \ref{thmimp}}

The condition $\theta(n)\ll n l(n)$ implies that $B(x)\ll x/\log x$, by \cite[Prop. 1]{PW} (with $y=x, z=1$).
This upper bound for $B(x)$ is used throughout this proof.
In Lemma \ref{funceq}, with $f(n)=1$, we estimate the last sum with Lemma \ref{PhiLemma}.
The contribution from the error term of Lemma \ref{PhiLemma} is $\ll x/\log^2 x$. 

To estimate the contribution form the penultimate term of Lemma \ref{PhiLemma}, 
note that $\log \theta(n) \asymp \log (n \theta(n))$. We write 
\begin{equation}\label{firsteq2}
\sum_{n\in \mathcal{B}(x)\atop n \theta(n) \le x}  \frac{\theta(n)}{\log \theta(n)} 
\asymp \sum_{k\ge 0} \frac{x}{2^k \log (x/2^k)}\sum_{n\in \mathcal{B}(x)\atop n \theta(n) \in I_k } \frac{1}{n},
\end{equation}
where $I_k=(x/2^{k+1},x/2^k]$. The contribution from $k$ with $2^k>\sqrt{x}$ is trivially $O(x/\log^2 x)$,
so we may assume $2^k \le \sqrt{x}$ and $\log(x/2^k) \asymp \log x$. 
The conditions on $\theta$ mean that $n\theta(n) \in I_k$ implies that 
$n\in (M_k, N_k]$ where $N_k=\sqrt{x 2^{-k}}$ and $\log (N_k/M_k) \ll \log f(x)$. 
Partial summation yields
$$
\sum_{n\in \mathcal{B}(x)\atop n \theta(n) \in I_k } \frac{1}{n}
\le\frac{B(N_k)}{N_k} + \int_{M_k}^{N_k} \frac{B(y)}{y^2} dy 
\ll \frac{1}{\log N_k}+ \log\log N_k - \log\log M_k .$$
Writing $\log M_k = \log N_k - \log N_k/M_k = \log N_k +O(\log f(x))$, we find that the last expression 
is $O( (\log f(x))/\log x)$. Inserting this into \eqref{firsteq2}, we conclude that the 
contribution from the penultimate term of  Lemma \ref{PhiLemma} is $O(x (\log f(x))/\log^2 x)$. 

For the contribution from the term $x \prod_{p\le y} (1-1/p)$ in Lemma \ref{PhiLemma},
we have  
$$
\sum_{n\in \mathcal{B}} \frac{1}{n} \prod_{p\le \theta(n)} \left(1-\frac{1}{p}\right) = 1,
$$
by \cite[Thm. 1]{SPA}.
With Mertens' formula and the fact that $\omega(u)=0$ for $u<0$, we obtain
$$
B(x) = x \sum_{n\in \mathcal{B}} \frac{1}{n\log \theta(n)} 
\left\{e^{-\gamma} - \omega\left(\frac{\log x/n}{\log \theta(n)}\right)\right\}
+O\left(\frac{x\log f(x)}{\log^2 x}\right).
$$
We now replace each instance of $\theta(n)$ by $2n$, incurring an error of 
$$
\ll \frac{x\log f(x)}{\log^2 x} + x\int_x^\infty \frac{\log f(y)}{y\log^3 y} dy \ll \frac{x\log f(x)}{\log^2 x} ,
$$
since $\log f(y) / \log y$ is eventually decreasing.  
Next, we use partial summation on the last sum, with an error $\ll x/\log^2 x$, to obtain
$$
B(x) = x \int_{1}^\infty \frac{B(y)}{y^2\log 2y} 
\left\{e^{-\gamma} - \omega\left(\frac{\log x/y}{\log 2y}\right)\right\} dy
+O\left(\frac{x\log f(x)}{\log^2 x}\right).
$$
The rest of the proof is the same as that of \cite[Thm. 5.1]{PDD}. 
In the process of inverting a Laplace transform, two additional error terms appear. They are 
$$\ll \frac{x}{\log^3 x} \int _{1}^x \frac{\log f(y)}{y} dy + \frac{x}{\log x} \int_x^\infty \frac{\log f(y)}{y \log^2 y} dy
\le \frac{x \log f(x)}{\log^2 x} + \frac{x E(x)}{\log x}.
$$
Since $f(x)$ is non-decreasing, $f(x)\ge f(1)\ge 2$ and $E(x) \ge (\log f(x))/\log x \ge \log 2 / \log x$, so that 
all error terms are $O( x E(x)/\log x)$.  

The proof of the formula for $c_\theta$ is identical to that of \cite[Thm. 1]{CFAE}, where this formula is 
derived under a more restrictive upper bound condition on $\theta(n)$. This stronger condition was needed in \cite{CFAE},
only because \cite[Thm. 5.1]{PDD} required it to ensure that $B(x)\sim c_\theta x/\log x$ holds.

\end{document}